\documentclass[oneside,english]{amsart}
\usepackage[T1]{fontenc}
\usepackage[latin9]{inputenc}
\usepackage{verbatim}
\usepackage{float}
\usepackage{amsfonts}
\usepackage{amsthm}
\usepackage{amssymb}
\usepackage{graphicx}
\PassOptionsToPackage{normalem}{ulem}
\usepackage{ulem}
\usepackage{fullpage}

\usepackage{bbm}
\usepackage{bm}
\usepackage{dsfont}

\usepackage{longtable}

\usepackage{times}
\usepackage{color,cite,graphicx}
\fussy
\usepackage{chapterbib}

\def\[{\begin{equation}}
\def\]{\end{equation}}
\newcounter{numer1}

\usepackage{multirow}
\usepackage{setspace}

\usepackage{graphicx}
\usepackage{amssymb}
\usepackage[all]{xy}

\usepackage{epsfig}

\newtheorem{theorem}{Theorem}
\newtheorem{lemma}{Lemma}
\newtheorem{proposition}{Proposition}

\newtheorem{definition}{Definition}

\newtheorem{OBS}{Remark}
\newtheorem{exmp}{Example}
\newtheorem*{notation*}{Notation}
\newtheorem{notation}{Notation}

\usepackage{color}
 \usepackage[foot]{amsaddr}

\title{The Effects of Structural Perturbations on the Synchronizability of Diffusive Networks}
\author{Jan Philipp Pade$^{1}$, Camille Poignard$^{2}$ and Tiago Pereira$^{2}$}

\address{$^1$ Institut f\"ur Mathematik,  Humboldt Universit\"at zu Berlin, Berlin, Germany  \\
$^2$ Instituto de Ciencias Matem\'aticas e Computa\c{c}\~ao,   Universidade de S\~ao Paulo, S\~ao Carlos,  Brasil}
\email{$^1$ pade@math.hu-berlin.de}

\begin{document}
\maketitle

\begin{abstract}
We investigate the effects of structural perturbations of both, undirected and directed diffusive networks on their ability to synchronize.  We establish a classification of directed links according to their impact on synchronizability.  We focus on adding directed links in weakly connected networks having a strongly connected component acting as driver. When  the connectivity of the driver is not stronger than the connectivity of the slave component, we can always make the network strongly connected while hindering synchronization. On the other hand, we prove the existence of a perturbation which makes the network strongly connected while increasing the synchronizability. Under additional conditions, there is a node in the driving component such that adding a single link starting at an arbitrary node of the driven component and ending at this node increases the synchronizability.

\end{abstract}
\section{Introduction}

Synchronization is an important  phenomenon in real world networks. For instance, in power-grids, power-stations must work in 50Hz-synchrony in order to avoid blackouts \cite{Doerfler2012,Motter2013}. In sensor networks, synchronization among the sensors is vital for the transmission of information \cite{Papadopoulos2005,Poonam2017}. On the other hand, synchronization of subcortical brain areas such as in the Thalamus is strongly believed to be the origin of motor diseases such as Dystonia and Parkinson \cite{Hammond2007,Milton2003,Starr2005}. In all of the mentioned examples, the stability of synchronous  states is crucial for the network's function or dysfunction respectively. Motivated by these observations, stability properties of synchronous states in systems of coupled elements have been investigated intensively \cite{Barahona2002,Pecora1998,Pikovsky2001,Field2017}. 

An important class, mimicking the above examples, is given by networks of identical elements which are coupled in a diffusive manner. That is, networks for which the dynamics of a node  depend  on the difference between its own state and its input. A special focus has been on unraveling the connection between such a network's coupling topology and its overall dynamics \cite{Matteo,Jalili2013,Motter2005,Nishikawa2006a,Motter2003,Nishikawa2017,agarwal2010,bick2017}.
 
While certain correlations have been observed, in general it remains unclear how to relate the structure to dynamical properties such as the stability of synchrony. A particularly interesting and important question in this category is the following: Assume that link in a network is perturbed or a new link with a small weight is added to a network. What is the impact on the dynamics? For instance, in interaction graphs of gene networks, it has been shown that adding links between two stable systems can lead to dynamics with positive topological entropy \cite{Poignard2013}. In diffusive systems such as laser networks, it was shown that the addition of a link can lead to synchronization loss \cite{pade2015,pade2015a}. In this article, we focus on the question whether these structural perturbations lead to higher or lower synchronizability. We give rigorous answers to this question, for both undirected and directed networks. Let us first introduce the model and motivate the main questions with some examples.

\subsection{Model and Examples} 

We call a network a triplet  $(\mathcal{G}, \bm{f}, \bm{H})$, where $\mathcal{G}$ is a graph, possibly weighted and directed, $\bm{f}:\mathbb{R}^{\ell}\rightarrow\mathbb{R}^{\ell}$ is a function representing the local dynamics of each node, and $\bm{H}$ is a coupling function between the nodes of the network. If no confusion can arise, we sometimes identify a network and its underlying graph $\mathcal{G}$. To this triplet we associate the following coupled equations.
\begin{equation}\label{eq:SystemAkinToDiff}
\dot{\bm{x}}_{i}=\bm{f}(\bm{x}_{i})+\alpha\sum_{j=1}^{N}W_{ij}\bm{H}(\bm{x}_{j}-\bm{x}_{i})\qquad i=1,2,\ldots,N.
\end{equation}
Here, $\alpha\geq0$ is the overall coupling strength 
and $\bm{W}=[W_{ij}]_{1\leq i,j\leq N}\in\mathbb{R}^{N\times N}$ is the adjacency matrix associated with the graph $\mathcal{G}$. In other words, $W_{ij}\geq0$ measures the strength of interaction from node $j$ to node $i$. 
The network $(\mathcal{G}, \bm{f}, \bm{H})$ is undirected if $\bm{W}$ is symmetric, otherwise it is directed.
The theory we develop here can include networks of non-identical elements with minor modifications \cite{Pereira2014}. 

Note that due to the diffusive nature of the coupling, if all oscillators start with the same initial condition,  then the coupling term vanishes identically. This ensures that the globally synchronized state is an invariant state for all coupling strengths $\alpha$, and we call the set
\begin{equation}\label{eq:sync-manifold}
M := \big\{ x_i \in U \subset \mathbb{R}^{\ell}   \mbox{ for } i \in \{ 1, \cdots , N \}  : x_1=\cdots = x_N \big\}
\end{equation}
the \emph{synchronization manifold}. The transverse stability of $M$ depends on the structure of the graph $\mathcal{G}$. Indeed, structural changes in $\mathcal{G}$ can have a drastic influence on the stability of $M$ as can be seen in the next examples which serve as motivating examples for the subsequent analysis.

\subsection{Structural Perturbations in Undirected Networks -- A Partition for Synchronization Improvement} Adding undirected links to an undirected network can never hinder synchronization, we state this result in Proposition \ref{lem:FiedlerGrowth}.  However, adding directed links to an undirected network can  either facilitate or hinder synchronization as the following example reveals. 
Consider the network with a graph $\mathcal{G}$ as shown in Figure \ref{Example2}. This example can be decomposed into two connected subgraphs (in blue and red) such that the addition of any directed connection between the two subnetworks facilitates synchronization - regardless of its direction.
\begin{figure}[H]
\includegraphics[scale=0.4]{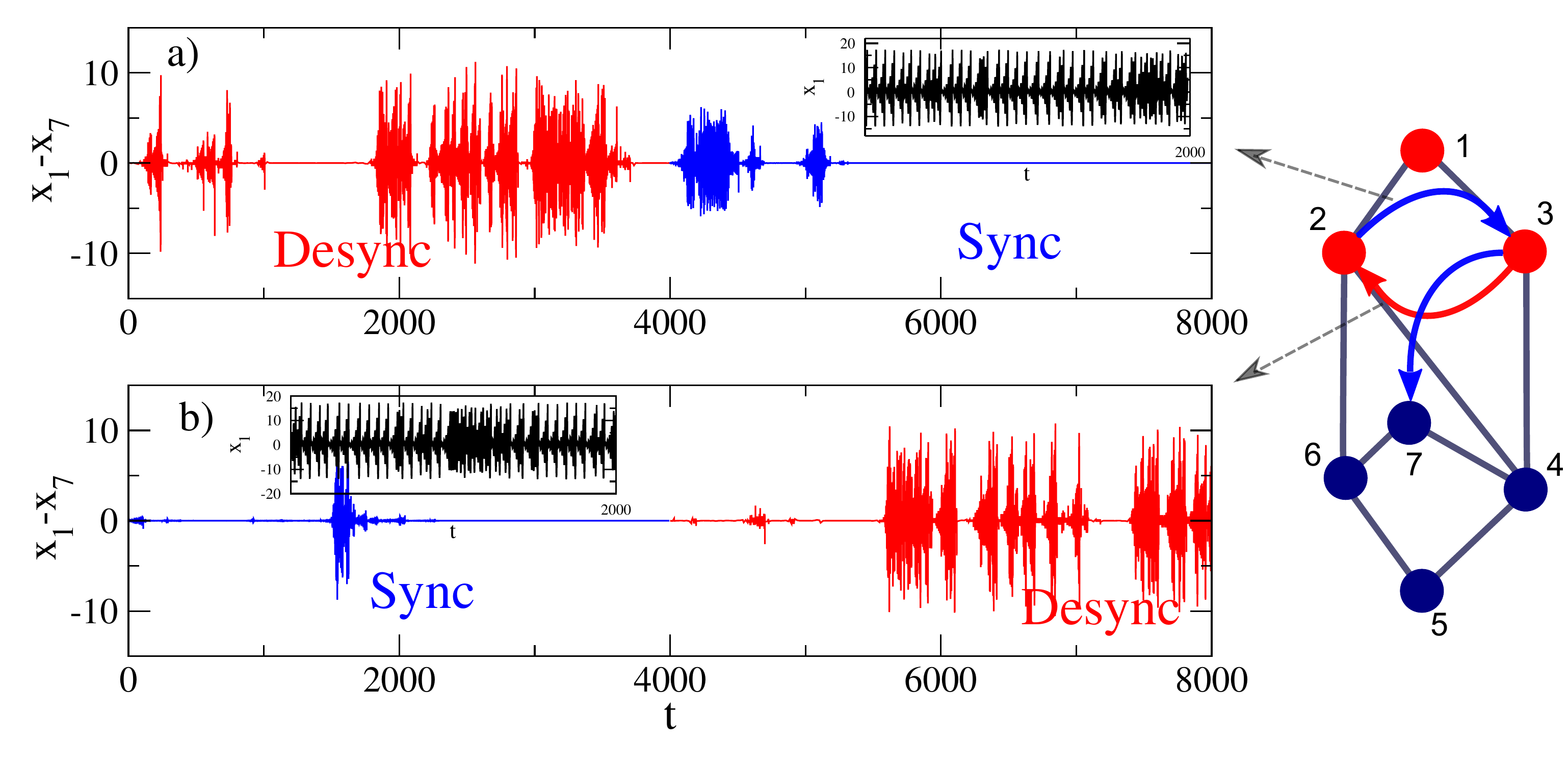}\caption[~~~~Dissection of graphs.]{Partition of a graph corresponding to the dynamic role of the nodes. Connecting red nodes and blue nodes with a directed link such as $3\rightarrow 7$ favours synchronization - regardless of the direction of the added link. Adding links between nodes of the same color has opposite effects on synchronizability. Simulations show dynamics of the network on the right with local dynamics $\bm{f}$ given by R\"ossler systems in the chaotic regime. In a), the addition of the blue link connecting nodes $2$ and $3$ at time $t=4000$ leads to chaotic synchronization of the whole network. In b), the addition of the reverse link in red destabilizes the synchronous state. The insets show time traces of the chaotic dynamics of a single node.}
\label{Example2}
\end{figure}

On the other hand, the impact of increasing weights or adding links with small weights between red nodes does depend on the direction of the perturbation. More precisely, the addition of a link from node $i$ to node $j$ has the opposite effect on the stability of synchronization as the addition of a link from node $j$ to node $i$. The same is valid for the blue nodes. For example, in Figure \ref{Example2} adding the directed blue link (with small weight) will enable  the network to synchronize its previously unsynchronized nodes (Plot a)), whereas adding the red link will hinder synchronization (Plot b)). In our first main result (Theorem \ref{Thm_Dissec_of_graphs}) we show that this situation is generic. What's more, we establish a full classification of directed links according to their impact on synchronizability.

\subsection{Structural Perturbations in Directed Networks -- About Masters and Slaves}
Directed networks always consist of one or several strongly connected subnetworks in which every node is reachable from any other node through a directed path. If there is more than one strongly connected subnetwork, two such subnetworks can be connected through unidirectional links pointing from one subnetwork to another. In the top right of Figure \ref{Fig_example1}, we show a network composed of two strongly connected subnetworks (without the red link), which is weakly connected: starting from the smallest connected subnetwork it is not possible to reach the larger connected subnetwork through a directed path. In the physics literature, this configuration is called {\it master-slave} coupling as the subnetwork consisting of nodes $1$, $2$ and $3$ drives the subnetwork consisting of nodes $4$ and $5$. 
This master-slave configuration is believed to have many optima such as synchronization. For instance, feedforward networks can synchronize for a wide range of coupling strengths while having only a few links \cite{Nishikawa2006a}. The network presented in Figure  \ref{Fig_example1} also supports stable synchronized dynamics.
An important question concerns the network dynamics once we make qualitative structural changes. For instance, what happens if the we add a link breaking the master-slave configuration?
\begin{figure}[H]
\includegraphics[scale=0.35]{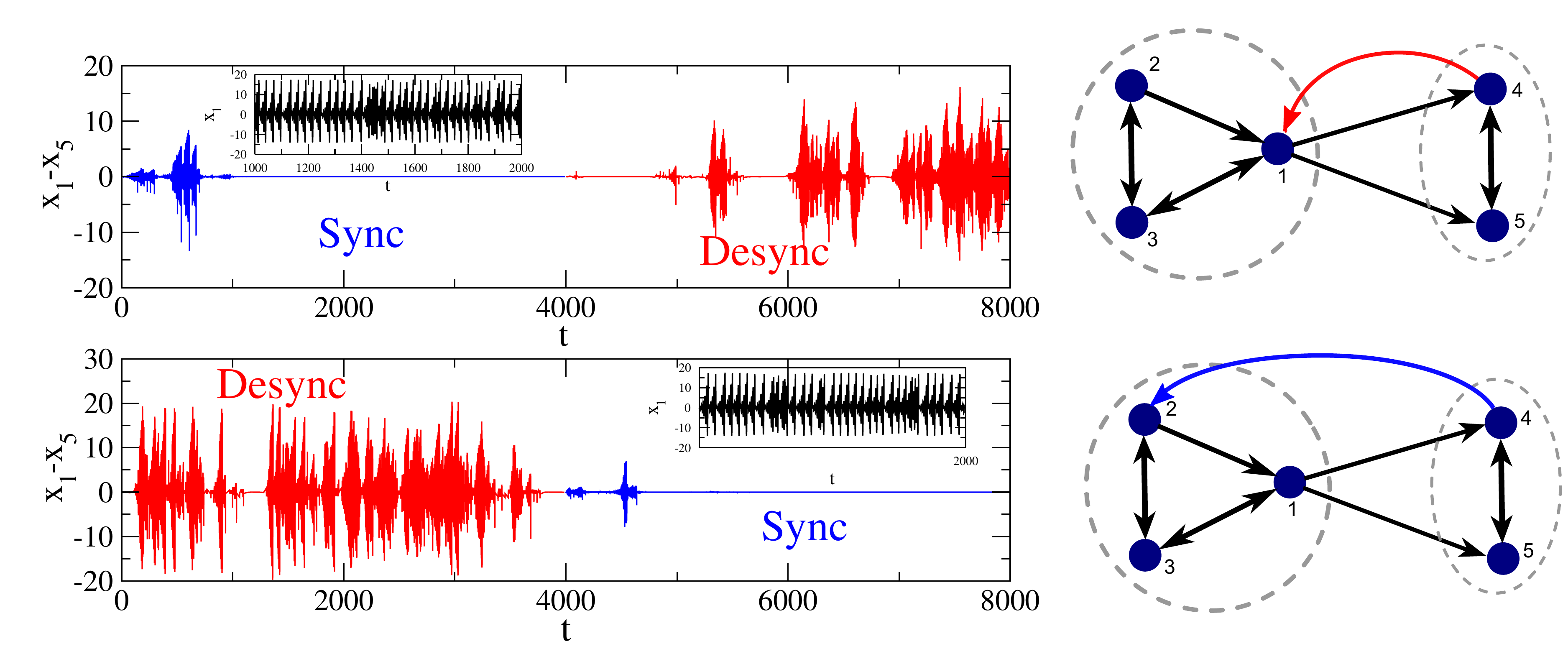}\caption[~~~~Impact of Making the network strongly connected.]{Simulations of the networks shown on the right. The local dynamics $\bm{f}$ are given by R\"ossler systems. In the top plot, the red link is added after time $t=4000$ and destroys the master-slave configuration by making the network strongly connected. As a consequence, the previously stable chaotic synchronization is no longer supported. In the bottom row, parameters are adjusted such that synchronization is unstable for the original network. The addition of the blue link at time $t=4000$ makes the network strongly connected again. However, in this case it leads to a stabilization of the synchronous state.}
\label{Fig_example1}
\end{figure}
An example for this is found in Figure \ref{Fig_example1}a). Introducing the new link (in red) makes the whole network strongly connected: there is a directed path connecting any two vertices in the network. Therefore, the addition of the link significantly improves the connectivity properties of the network. However, this structural improvement has a surprising consequence for the dynamics: the network synchronization is lost, as can be seen in the simulation in Figure \ref{Fig_example1}a).

{\it Hinderance of synchronization is not about breaking a master-slave configuration}. One may think that this synchronization loss appears because we are breaking the master-slave configuration. This rationale is justified as master-slave configurations are known to synchronize well \cite{Nishikawa2006a}. However, the synchronization loss is not related to the master-slave breaking. Indeed, adding a different connection which also makes the network strongly connected stabilizes the synchronous state (see Figure \ref{Fig_example1}b)).

{\it Hinderance of synchronization is not about reinforcing the hub.} Synchronization loss in the example of Figure  \ref{Fig_example1}a) appeared as an additional link was added to the hub of the largest subnetwork (the most connected node in the network). However, running experiments on random graphs with hubs, we found several  counter-examples in which linking to the hub improves synchronization.

To sum up, while in some settings, master-slave configurations and the presence of hubs play an important role for the behaviour of a network under structural perturbations \cite{Matteo,Pereira2010,Belykh2005}, for networks with diffusive dynamics given by Equation (\ref{eq:SystemAkinToDiff}) the whole picture is more complex. Adding extra links generates nonlinearities which can either enhance or hinder synchronization. Our second main result (Theorem \ref{thm:CutsetPerturbation}) gives an almost complete explanation of the complex behaviour of such weakly connected directed networks when a master-slave configuration is reinforced or destroyed respectively.

\subsection{Informal Statement of Results}\label{sec:informal}


Using the master stability approach to tackle the transverse stability of the synchronization manifold $M$  \cite{Pereira2014} 
we can in fact reduce the stability problem to the spectral analysis of graph Laplacians. The rather mild assumptions needed for this approach are specified in Section \ref{sec:assumptions}.  Let us now give an informal statement of our main results. \\

\noindent
{\it Informal Statement of Theorem 1 (Partitioning of Undirected Graphs)}. Consider a weighted connected undirected network with graph $\mathcal{G}$. We prove that typically, $\mathcal{G}$ can be uniquely decomposed  into two disjoint connected subgraphs $\mathcal{G}_1$ and $\mathcal{G}_2$ containing all the nodes of $\mathcal{G}$ such that increasing weights or adding a directed link with small weight between $\mathcal{G}_1$ and $\mathcal{G}_2$ facilitates synchronization, independent of the link's direction. Adding links among nodes of either one of the two subgraphs has  opposite effects on the stability: more precisely, if adding a link from node $k$ to node $j$ facilitates synchronization, adding a link from $j$ to $k$ is detrimental for synchronization. Finally, there exists a hierarchy of connected subnetworks building up with nodes from $\mathcal{G}_1$, $\mathcal{G}_2$ respectively, such that adding a link between two such subnetworks decreases the synchronizability.\\

In Theorem 2 we consider networks consisting of two strongly connected components. The general case of a higher number of strongly connected components is a straightforward generalization.\\
\noindent
{\it Informal Statement of Theorem 2 (Breaking Master-Slave configurations)}. Consider a directed network connected in a master-slave configuration as in Figure \ref{Fig_example1}.
First, consider the situation where {\it the master network is poorly connected}. Then, strengthening the cutset is immaterial for synchronization and it will neither facilitate nor hinder synchronization. Second, consider the case where the {\it master network is highly connected}. In this case we have:

\begin{itemize}
\item[] {\it -- Strengthening the driving facilitates synchronization}, leading to shorter transients towards synchronization and augmenting the basin of attraction.  

\item[] {\it -- Master-Slave configurations are non-optimal}. It is always possible to break the master-slave configuration in a way that favours synchronization (e.g. Figure \ref{Fig_example1}). Provided the overall connectivity of the network is poor it is even possible to find one or several nodes in the master-component such that the addition of an {\it arbitrary single link} ending at this node and breaking the master-slave configuration increases the synchronizability. In fact, if additionally, the Laplacian of the master component has zero column-sums, then any perturbation in opposite direction of the cutset enhances the synchronizability. 

\item[] {\it --  Breaking Master-Slave configurations can hinder Synchronization.}
If the connectivity of the master component is not much stronger than the connectivity of the slave component (a precise condition of this will be given in Theorem 2), we can always find a cutset such that there is a perturbation in opposite direction of this cutset for which synchronization is hindered.
%
Our result reveals the role the eigenvectors of the master network play in the destabilization of the synchronous motion. For example, if $\alpha_k$ is an eigenvalue of the Laplacian of the master and close to $\lambda_2$ the spectral gap in the slave component (this is the case in our illustration), then the eigenvector $\bm{X}_k$ associated to $\alpha_k$ encodes the important information about the possible destabilization. 
For instance, assume that the $i$th entry of $\bm{X}_k$ is the maximal (or minimal) one.  If the slave network is driven by a  links coming from the $i$th node then it is possible to destabilize the synchronization.
\end{itemize}



The remainder of the article is organized as follows. In Section \ref{sec:Stab_Sync} we introduce basic notions concerning the stability of synchronization in networks and present the notion of synchronizability of networks. In Section \ref{main} we present the two main results of our paper, followed by
Section \ref{sec:Stab_Undir} in which we prove Theorem \ref{Thm_Dissec_of_graphs} and other results (Proposition \ref{lem:FiedlerGrowth} and Proposition \ref{30June2017}) completing the study on the dynamical impact of perturbations in undirected networks.
Section \ref{sec:proof_cutset} is devoted to the proof of Theorem \ref{thm:CutsetPerturbation}. The article concludes with a discussion in Section \ref{sec:discussion}.

{\bf Acknowledments:} We are in debt with Mike Field, Chris Bick, Anna Dias, Jeroen Lamb, Matteo Tanzi, and  Serhyi Yanchuk for useful discussions. TP and CP were partially supported by FAPESP Cemeai grant 2013/07375-0, Russian Science Foundation grant
14-41-00044 at the Lobayevsky University of Nizhny Novgorod and by the European Research Council [ERC AdG grant number 339523 RGDD]. JPP was supported by the DFG Research Center Matheon in Berlin.

\section{Notations and Definitions}\label{sec:Stab_Sync}
\subsection{Weighted Graphs and Laplacian Matrices}
We consider networks of identical elements with diffusive interaction. It will be useful to interpret the coupling structure of the network as a graph. We recall some basic facts on graph theory.

\begin{definition}[Weighted graphs]\label{def:graph}
\noindent
A weighted graph $\mathcal{G}$ is a set of nodes $\mathcal{N}$ together with a set of edges $\mathcal{E}\subset\mathcal{N}\times\mathcal{N}$ and a weight function $w:\mathcal{E}\rightarrow\mathbb{R}_+$. We say that the graph is unweighted when we have $w(i,j)=1$ for all $(i,j)$ in $\mathcal{E}$. Moreover
\begin{itemize}
\item[(i)] We say that the graph is undirected if  $(i,j)\in\mathcal{E}\iff (j,i)\in\mathcal{E}$ and $w(i,j)=w(j,i)$ for all  $(i,j)\in\mathcal{E}$. Otherwise, the graph is directed and edges are assigned orientations. A directed graph is also called digraph.
\item[(ii)] $\mathcal{G}=(\mathcal{N},\mathcal{E},w)$ is a subgraph of $\mathcal{G}^{\prime}=(\mathcal{N}^{\prime},\mathcal{E}^{\prime},w^{\prime})$ if $\mathcal{N}\subseteq \mathcal{N}^{\prime}$, and $\mathcal{E}\subseteq\mathcal{E}^{\prime}$. 
In this case, we write $\mathcal{G}\subseteq\mathcal{G}^{\prime}$.
\item[(iii)] The adjacency matrix $\bm{W}\in\mathbb{R}^{N\times N}$ of the graph $\mathcal{G}$ is defined through 
$$
W_{ij}=\left\{\begin{array}{cc}
w(i,j) & \text{if } (i,j)\in\mathcal{E}\\
0 & else
\end{array}\right.
$$
\end{itemize}
\end{definition}
To deal with synchronization of networks, we will focus on graphs exhibiting some sort of connectedness.
 
\begin{definition}[Connectedness of graphs]
\label{Def:graph-connection}
 An undirected graph $\mathcal{G}$ is connected if for any two nodes $i$ and $j$, there exists a path $\{i=i_1,\cdots ,i_p=j\}$ of nodes (successively connected by edges of $\mathcal{G}$) between node $i$ and node $j$. For directed graphs we have two notions of connectedness
 \begin{itemize} 
\item[(i)] A digraph $\mathcal{G}$ is strongly connected if every node is reachable from every other
node through a directed path.
\item[(ii)] The digraph is weakly connected if it is not strongly connected and the underlying graph which is obtained by ignoring the links' directions is connected. A maximal strongly connected subgraph of a weakly connected digraph is called strongly connected component, or strong component. The maximal set of links which connects two strong components is called cutset.
\item[(iii)] A spanning diverging tree of a digraph is a weakly connected subgraph such that one node, the root node, has no incoming edges and all other nodes have one incoming edge.
\end{itemize}
\end{definition}
Let a weighted digraph be given by its adjacency matrix $\bm{W}$ and let $\bm{D_W}$ be the diagonal matrix whose $i$-th entry is given by the degree $d_i=\sum_{j=1}^NW_{ij}$  of node $i$. The {\it Laplacian} of $\bm{W}$ is then defined as 
\[
\bm{L_W} = \bm{D_W}-\bm{W}
\]
As the Laplacian has zero row-sums, the vector $\mathbf{1}$ (all entries equal to $1$) is an eigenvector associated with the eigenvalue $0$.  
In virtue of the  Gershgorin theorem \cite{Horn1985}, the remaining eigenvalues $\lambda_{i}$ have nonnegative real parts. In what follows we will always assume that the eigenvalues are ordered in the following way
\[
0=\lambda_{1}\leq\Re\left(\lambda_{2}\right)\leq\ldots\leq\Re\left(\lambda_{N}\right).
\]
This allows us to introduce a standard notation from algebraic graph theory.
We call the second eigenvalue $\lambda_2=\lambda_2(\bm{L_W})$ of $\bm{L_W}$ the {\it spectral gap}. If the graph is undirected, we call the corresponding normalized eigenvector the {\it Fiedler vector} \cite{Chung1997}.

\subsection{Synchronizability of Networks:  Assumptions}\label{sec:assumptions}
Although equations of the form (\ref{eq:SystemAkinToDiff}) are heavily used in the context of network synchronization, it was only very recently that a stability result has been established for the general case of time-dependent solutions \cite{Pereira2014}. In order to guarantee for the stability of synchronous motion we make the following assumptions:  \\

\noindent
\textbf{A1 (Structural assumption)} $\mathcal{G}$ has a spanning diverging tree. \\
\noindent
\textbf{B1 (Absorbing Set) } The vector field $\bm{f}:\mathbb{R}^{\ell}\rightarrow\mathbb{R}^{\ell}$ is
continuous and there exists a bounded, positively invariant open set $U\subset\mathbb{R}^{\ell}$ such that $\bm{f}$ is continuously differentiable
in $U$ and there exists a $\varrho>0$ such that
\[
\left\Vert d\bm{f}\left(\bm{x}\right)\right\Vert \leq\varrho\qquad\forall \bm{x} \in U.
\]
\textbf{B2 (Smooth Coupling)} The local coupling function $\bm{H}$ is smooth satisfying
$\bm{H}(\bm{0})=\bm{0}$ and the eigenvalues $\beta_{j}$ of $d\bm{H}\left(\bm{0}\right)$
are real.\\~\\
\textbf{B3 (Spectral Interplay)} The eigenvalues $\beta_{j}$ of $d\bm{H}\left(\bm{0}\right)$ and $\lambda_{i}$ of $\bm{L}$ fulfil 
\begin{equation}\label{def:gamma}
\gamma:=\Re(\lambda_2) \min_{1\leq j\leq \ell}\beta_j>0.
\end{equation} 

Let us shortly discuss these assumptions to see that they are somehow natural. \textbf{A1} concerns the coupling topology of the underlying (directed) graph. For undirected networks, it simply amounts to assuming that the underlying coupling graph is connected.
In the case of a weakly connected digraph consisting of several strong components it is equivalent to the fact that there is at most one {\it root-component}: a strong component which does not have any incoming cutsets.  Algebraically, a consequence of this assumption for both, undirected and directed graphs, is that the zero eigenvalue of the graph Laplacian is simple \cite{agaev2000}. 

Assumption \textbf{B1} guarantees that the nodes' dynamics admit an invariant compact set, for instance an equilibrium, a periodic orbit or a chaotic orbit.

The second dynamical condition \textbf{B2} guarantees that the synchronous state ${x}_{1}\left(t\right)={x}_{2}\left(t\right)=\cdots={x}_{N}\left(t\right)$ is a solution of the coupled equations for all values of the overall coupling strength $\alpha$: when starting with identical initial conditions the coupling term vanishes and all the nodes behave in the same manner.

For the last condition \textbf{B3} remark that for undirected graphs the zero eigenvalue of the graph Laplacian is non-simple iff the underlying graph is disconnected \cite{Brouwer2011}. In this case the stability condition would be violated. Indeed, in order to observe synchronization it is clear that one should consider networks which are connected in some sense. We remark that the assumption that the $\beta_j$ are real is true for many applications. The general case of complex eigenvalues $\beta_j$ can be tackled in a similar way, but the analysis becomes more technical without providing new insight into the phenomena \cite{Pereira2014}.

\subsection{Critical Threshold for Synchronization}\label{sec:crit_threshold}

Under the previous assumptions it was shown in \cite{Pereira2014}  that for Equation (\ref{eq:SystemAkinToDiff}),  there exists an $\alpha_c = \alpha_c(\mathcal{G},\bm{f},\bm{H})$ such that if the global coupling strength fulfils $\alpha > \alpha_c$ the network is locally uniformly synchronized:  the synchronization manifold attracts uniformly in an open neighborhood. More precisely, there exists a $C=C\left(\bm{L},d\bm{H}\left(\bm{0}\right)\right)>0$
such that if the initial condition $\bm{x}_{i}\left(t_{0}\right)$ is in a neighborhood of the synchronization manifold, then the solution $\bm{x}(t)$ of Equation (\ref{eq:SystemAkinToDiff}) fulfils
$$
\left\Vert {x}_{i}\left(t\right)-{x}_{j}\left(t\right)\right\Vert \leq Ce^{-\left(\alpha\gamma-\rho\right)\left(t-t_{0}\right)}\left\Vert {x}_{i}\left(t_{0}\right)-{x}_{j}\left(t_{0}\right)\right\Vert \qquad\forall t\geq t_{0}.
$$
Now, the key connection to the graph Laplacian is that the critical coupling $\alpha_c$ can be factored as
\begin{equation}
\alpha_c=\frac{\rho}{\gamma}\label{eq:stabCond1}
\end{equation}
where $\rho=\rho(\bm{f},d\bm{H}(\bm{0}))$ is a constant depending only on $\bm{f}$ and $d\bm{H}(\bm{0})$. So the constant $\gamma$ which represents the coupling structure (see Equation (\ref{def:gamma})) is directly related to the contraction rate towards the synchronous manifold. In fact,  the condition $\alpha>\alpha_c$ for stable synchronous motion now writes as
\begin{equation}
\alpha\Re\left(\lambda_{2}\right)\min_{1\leq j\leq m}\beta_{j}>\rho.\label{eq:StabCond2}
\end{equation}
Condition (\ref{eq:StabCond2}) shows that the spectral gap $\lambda_{2}$ plays a central role for synchronization properties of the network. 

\subsubsection{Measures of Synchronization}\label{measure} We can use the critical coupling $\alpha_c$ in order to define a measure of synchronizability.
\begin{definition}
We say that the network $(\mathcal{G}_1, \bm{f}_1, \bm{H}_1)$ is more synchronizable than $(\mathcal{G}_2, \bm{f}_2, \bm{H}_2)$ 
if their critical couplings satisfy
\[
\alpha_c( \mathcal{G}_1, \bm{f}_1, \bm{H}_1) < \alpha_c( \mathcal{G}_2, \bm{f}_2, \bm{H}_2).
\] 
\end{definition}
Indeed, the range of coupling strengths which yield stable synchronization is larger for $( \mathcal{G}_1, \bm{f}_1, \bm{H}_1)$. Fixing the dynamics $\bm{f}$ and the coupling function $\bm{H}$ we can now measure whether structural changes in the graph will favour or hinder synchronization. Assume we have a network $(\mathcal{G}, \bm{f}, \bm{H})$ and a perturbed network $(\tilde{\mathcal{G}}, \bm{f}, \bm{H})$ with corresponding spectral gaps $\lambda_2$ and $\tilde{\lambda}_2$. \\
{\it A direct consequence of the definition of synchronizability is that if $\Re(\lambda_2)<\Re(\tilde{\lambda}_2)$, the perturbed network $(\tilde{\mathcal{G}}, \bm{f}, \bm{H})$ is more synchronizable than $(\mathcal{G}, \bm{f}, \bm{H})$}.\\
We also say that the modification {\it favours synchronization}. Otherwise, if  $\Re(\lambda_2)>\Re(\tilde{\lambda}_2)$ we say the structural perturbation {\it hinders synchronization}. This enables us to reduce the stability problem to an algebraic problem, i.e. the behaviour of the spectral gap under structural perturbations. We will use this approach throughout the whole article.

\section{Main Results}\label{main}

In this section we state our two main results: Theorem \ref{Thm_Dissec_of_graphs} deals with the undirected case, Theorem \ref{thm:CutsetPerturbation} with the directed case. 
%
We emphasize that, given assumption \textbf{A1}, both theorems are structurally generic, a term which we introduced in an earlier paper \cite{Poignard2017}.
In order to explain the notion of structural genericity, consider the set of Laplacians corresponding to networks with
identical coupling topologies but potentially different weights. In this set, the subset of Laplacians for which our results are valid
is dense and its complement is of zero Lebesgue measure. In other words, given any network topology satisfying \textbf{A1}, 
our results are valid up to a small perturbation of the weights of the \textit{existing} links of this network.
This structural genericity is stronger than the classical one for which it is usually necessary to perturb drastically the structure of the original network itself. 
For more details on structural genericity see Theorem 3.1 for the undirected case and Theorem 6.6 for the directed case in \cite{Poignard2017}.


\subsection{Structural Perturbations in Undirected Networks}\label{main1}

In this subsection we focus on rank-one perturbations which correspond to the addition of a link in a directed way or in an undirected way. This extends naturally to increasing the weight of an existing link. 
Observe that structurally more complex perturbations such as the addition or deletion of two or several links can be treated in the same way since by Lemma \ref{lem:EVGrowth}, the results we obtain are linear in the perturbation term.
Let us first introduce the following matrix corresponding to the aforementioned perturbations.

\begin{notation}
We denote by ${\bm{L}}_{{kl}}$ (respectively ${\bm{L}}_{{k \to l}}$) the Laplacian matrix of the disconnected undirected (respectively directed) graph with $n$ nodes having only one link between node $k$ and node $l$ (respectively pointing from $k$ to $l$):
\begin{equation}\label{PertMatrix}
\smaller{
{\bm{L}}_{kl}=\left(\begin{array}{ccccccc}
0 &  &  &  & & & \\
 &\ddots & & & & & \\
 & &1 &\ldots &-1& &  \\
 & & &\ddots & & & \\
 & & -1& \ldots & 1& & \\
 & & & & & \ddots&\\
 & & & & & &0
\end{array}\right)},
\smaller{
{\bm{L}}_{{k \to l}}=\left(\begin{array}{ccccccc}
0 &  &  &  & & & \\
 &\ddots & & & & & \\
 & &1 &\ldots &-1& &  \\
 & & &\ddots & & & \\
 & & & & \ddots& & \\
 & & & & & \ddots&\\
 & & & & & &0
\end{array}\right)}
\end{equation}
where the entry $-1$ in the matrix $\bm{L}_{k \to l}$ is in the $l$-th row and $k$-th column.
\end{notation}

Now, as exposed in Subsection \ref{measure}, for linear stability considerations it is sufficient to track the spectral gap of the perturbed graph Laplacian.

\begin{definition}[Motion of Eigenvalues]\label{motion}
For an undirected graph Laplacian $\bm{L_W}$ having a simple spectral gap $\lambda_2 \left(\bm{L_W}\right)\neq 0$, we set 
$$
s_{kl}:=  \lim_{\varepsilon \rightarrow 0 } \frac{1}{\varepsilon} \left( \lambda_2\left( \bm{L_W}+\varepsilon {\bm{L}}_{kl}  \right) - \lambda_2\left( \bm{L_W}  \right) \right)
$$
$$
s_{k \to l}:=  \lim_{\varepsilon \rightarrow 0 } \frac{1}{\varepsilon} \left( \lambda_2\left( \bm{L_W}+\varepsilon {\bm{L}}_{k\rightarrow l}  \right) - \lambda_2\left( \bm{L_W}  \right) \right)
$$
for the change rates of the spectral gap maps under the small structural perturbations $\varepsilon {\bm{L}}_{kl}$, $\varepsilon {\bm{L}}_{k\rightarrow l}$. 

\end{definition}
The maps $\varepsilon \mapsto \lambda_2\left( \bm{L_W}+\varepsilon {\bm{L}}_{kl}\right)$ and $\varepsilon \mapsto \lambda_2\left( \bm{L_W}+\varepsilon {\bm{L}}_{k\rightarrow l}\right)$
are smooth functions because of the simplicity of $\lambda_2 (\bm{L_W})$ at $\varepsilon=0$ (see Lemma \ref{lem:EVGrowth} below, in which we provide an expression of the derivative of the spectral gap map at $\varepsilon=0$).
This simplicity assumption is structurally generic, as shown in \cite{Poignard2017}.

As we show in Section \ref{sec:Stab_Undir}, adding undirected links to undirected networks will never hinder synchronization. The same is valid for increasing weights in undirected networks. This is not valid any more for directed perturbations of this type as our first main result shows.

\begin{theorem}\label{Thm_Dissec_of_graphs}
Let $\bm{L_W}$ be the Laplacian of an undirected, connected and weighted graph $\mathcal{G}$. Then, for a generic choice of the nonzero weights of $\mathcal{G}$ we have: 
\begin{description}
\item[\rm{(i) Improving synchronizability by a unique decomposition}] There is a unique partition of $\mathcal{G}$ into two
disjoint connected subgraphs $\mathcal{G}_1$ and  $\mathcal{G}_2$ 
such that for nodes $k$ and $l$ belonging to different subgraphs we have $s_{k\rightarrow l} > 0$.
\item[\rm{(ii) Opposite effects on synchronizability}] For nodes $k,l$ belonging to the same subgraph, so either $k,l\in\mathcal{G}_1$ or $k,l\in\mathcal{G}_2$ we have $s_{k\rightarrow l}s_{l\rightarrow k}\leq 0$.
\item[\rm{(iii)  Cascade of destabilization}] There are unique increasing sequences of connected subgraphs
\begin{eqnarray*}
\mathcal{G}_1&\subset & \mathcal{G}_{11}\subset\mathcal{G}_{12}\subset\ldots\subset\mathcal{G}_{1r}=\mathcal{G}\\
\mathcal{G}_2&\subset & \mathcal{G}_{21}\subset\mathcal{G}_{22}\subset\ldots\subset\mathcal{G}_{2m}=\mathcal{G}
\end{eqnarray*}
with $r+m\leq N$ such that for $j>i$ and for a node $k$ in $\mathcal{G}_{1i}\cap\mathcal{G}_2$ and a node $l$ in $\mathcal{G}_{1j}\setminus\mathcal{G}_{1i}$ we have $s_{l\rightarrow k} < 0$.
\end{description}
\end{theorem}

This result builds on a theorem by M. Fiedler that we recall in Section 3. 


\subsection{Structural Perturbations in Directed Networks}\label{section3.2}
In this section we investigate the class of directed networks satisfying assumption \textbf{A1} and 
consisting of at least two strong components. Due to \textbf{A1}, these networks have one root-component.
Furthermore, we restrict ourselves to the study of the dynamical role of links between strong components, 
i.e on cutsets. Here, perturbations can point either in direction of a cutset or in opposite direction of a cutset. 
For simplicity of presentation we can assume that there are only two strong components. 
So the corresponding Laplacian is of the form
\begin{equation}\label{eq:LaplacianCutset}
\bm{L_W}=\left(\begin{array}{cc}
\bm{L}_{1} & \bm{0}\\
-\bm{C}& \bm{L}_{2}+\bm{D}_{\bm{C}}
\end{array}\right),
\end{equation}
where $\bm{L}_{1}\in\mathbb{R}^{n\times n}$ and $\bm{L}_{2}\in\mathbb{R}^{m\times m}$ are the respective Laplacians of the strong components, $\bm{C}\in\mathbb{R}^{m\times n}$ is the adjacency matrix of the cutset pointing from one strong component to the other and $\bm{D_{C}}$ is a diagonal matrix with the row sums of $\bm{C}$ on its diagonal.
The results presented in Theorem \ref{thm:CutsetPerturbation} below can be generalized in a straightforward way to networks with more than two strong components. 
For instance, our results are still valid for graph Laplacians of the form
\begin{align*}
\bm{L_W}=\begin{pmatrix}
\bm{L}_{1} & \bm{0}& \bm{0} & \cdots & \bm{0}\\
-\bm{C_{21}} & \bm{L}_{2}+\bm{D}_{\bm{C_{21}}}& \bm{0} & \cdots &\bm{0}\\
-\bm{C_{31}} & -\bm{C_{32}} & \bm{L}_{3}+\bm{D}_{\bm{C_{31}}}+\bm{D}_{\bm{C_{32}}}&  0 \cdots &\bm{0}\\
\vdots &  & \ddots &\ddots& \vdots \\
\\
-\bm{C_{(p-1)1}}&\ldots &-\bm{C_{(p-1)(p-2)}}& \bm{L}_{p-1}+\sum_{i=0}^{p-2}\bm{D}_{\bm{C_{(p-1)i}}}& \bm{0}\\
\\
-\bm{C_{p1}} & \ldots && -\bm{C_{p(p-1)}}& \bm{L}_{p}+\sum_{i=0}^{p-1}\bm{D}_{\bm{C_{pi}}}\\
\end{pmatrix}.\\
\end{align*}
representing a graph with $p$ strong components connected by cutsets $\bm{C_{ij}}$.\\
We first remark that as a consequence of the block structure of $\bm{L_W}$ given by Equation \eqref{eq:LaplacianCutset}, its eigenvalues are either eigenvalues of $\bm{L}_{1}$ or eigenvalues of $\bm{L}_{2}+\bm{D}_{c}$. A structural perturbation in direction of the cutset induced by a nonnegative matrix $\bm{\Delta}\in\mathbb{R}^{m\times n}$ corresponds to the following modified Laplacian matrix
$$\bm{L}_p\left(\bm{\Delta}\right)=\left(\begin{array}{cc}
\bm{L}_{1} & \bm{0}\\
-\bm{C}-\bm{\Delta} & \bm{L}_{2}+\bm{D}_{\bm{C}+\bm{\Delta}}
\end{array}\right).
$$
A structural perturbation in opposite direction of the cutset induced by a nonnegative matrix $\bm{\Delta}\in\mathbb{R}^{n\times m}$ corresponds to the following modified Laplacian matrix
$$
\bm{L}_{p}\left(\bm{\Delta}\right)=\left(\begin{array}{cc}
\bm{L}_{1}+\bm{D}_{\Delta} & -\bm{\Delta}\\
-\bm{C} & \bm{L}_{2}+\bm{D_{C}}
\end{array}\right).
$$ 

\begin{OBS}
In the rest of the paper, to avoid cumbersome formulations, we will employ the formulation ``a structural perturbation $\bm{\Delta}$ in direction of the cutset" to refer to a structural perturbation in direction of the cutset induced by a nonnegative matrix $\bm{\Delta} \in\mathbb{R}^{m\times n}$". Similarly for structural perturbations in opposite direction of the cutset.
\end{OBS}

\begin{notation*}\label{motion2}
Given a Laplacian matrix $\bm{L_W}$ of a directed graph with simple spectral gap $\lambda_2\left(\bm{L_W}\right)$ and a nonnegative matrix $\bm{\Delta}$ we denote, similarly to the above notations, by 
$$
s\left(\bm{\Delta}\right):=  \lim_{\varepsilon \rightarrow 0 } \frac{1}{\varepsilon} \left( \lambda_2\left( \bm{L}_p(\varepsilon\bm{\Delta} )   \right) - \lambda_2\left( \bm{L_W}  \right) \right)
$$
the change rate of the spectral gap map under the small structural perturbations $\varepsilon \bm{\Delta}$ in direction or in opposite direction of the cutset.
\end{notation*}
Observe, as in Definition \ref{motion}, that the spectral gap map $\varepsilon \mapsto \lambda_2\left( \bm{L}_p(\varepsilon\bm{\Delta})\right)$ is regular because of the simplicity of $\lambda_2\left(\bm{L_W}\right)$ \cite{Horn1985}. In \cite{Poignard2017} we proved that having simple eigenvalues is a structurally generic property for graph Laplacians of weakly connected digraphs that satisfy Assumption \textbf{A1}.

Notice that we can possibly have $s\left(\bm{\Delta}\right) \in \mathbb{C}$ since the matrices involved in this notation are no more symmetric. However, we will prove in Section \ref{sec:proof_cutset} (Lemma \ref{lem:PositivOfEvectors}) that in the case where $\lambda_2\left(\bm{L_W}\right)$ is an eigenvalue of $\bm{L_{2}}+\bm{D_{C}}$, then $\lambda_2\left(\bm{L_W}\right)$ is real positive and therefore $s\left(\bm{\Delta}\right) \in \mathbb{R}$.
We can now state our second main result in the directed case:
\begin{theorem}\label{thm:CutsetPerturbation}
Let a directed graph $\mathcal{G}$ consist of two strong components connected by a cutset with adjacency matrix $\bm{C}$ and write the associated Laplacian as
\begin{equation*}
\bm{L_W}=\left(\begin{array}{cc}
\bm{L}_{1} & \bm{0}\\
-\bm{C} & \bm{L}_{2}+\bm{D_{C}}
\end{array}\right).
\end{equation*}
Assume \textbf{A1} is satisfied. Then, for a generic choice of the nonzero weights of $\bm{L_W}$ we have the following assertions:
\begin{description}

\item[\rm{(i)  {Invariance of synchronizability}}] If the spectral gap $\lambda_2$ of $\bm{L_W}$ is an eigenvalue of $\bm{L_1}$, then the network's synchronizability is invariant under arbitrary structural perturbations $\bm{\Delta}$
in direction of the cutset.

\item[\rm{(ii) {Improving synchronizability by reinforcing the cutset}}] If $\lambda_{2}$ is an eigenvalue of $\bm{L_{2}}+\bm{D_{C}}$, then the network's synchronizability increases for arbitrary structural perturbations $\bm{\Delta}$
in direction of the cutset.

\item[\rm{(iii) {Non-optimality of master-slave configurations}}] Assume $\lambda_{2}$ is an eigenvalue of $\bm{L_{2}}+\bm{D_{C}}$. Then we have the following statements:

\subitem {\rm (a)} There exists a structural perturbation $\bm{\Delta}$ in opposite direction of the cutset such that $s(\bm{\Delta})> 0$. 
\subitem {\rm (b)} There exists a constant $\delta (\bm{L}_1)>0$ and at least one node $1\leq k_0\leq n$ (in the driving component) such that if we have $0<\lambda_2<\delta (\bm{L}_1)$, then $s(\bm{\Delta})> 0$ for any structural perturbation $\bm{\Delta}$ consisting of only one link in opposite direction of the cutset and ending at node $k_0$.
\subitem {\rm (c)} If moreover $\bm{L}_1$ has zero column-sums, then there exists a constant $\delta (\bm{L}_1)>0$ such that if $0<\lambda_2<\delta (\bm{L}_1)$ we have $s(\bm{\Delta})> 0$ for any structural perturbation $\bm{\Delta}$ in opposite direction of the cutset.

\item[\rm{(iv) {Hindering synchronizability by breaking the master-slave configuration}}]
There exists a cutset $\bm{C}$ for which $\lambda_{2}$ is an eigenvalue of $\bm{L_{2}}+\bm{D_{C}}$ and a perturbation $\bm{\Delta}$ in opposite direction of $\bm{C}$ such that: if $\bm{L}_1$ admits a positive eigenvalue 
%
sufficiently small 
%
%
then we have $s(\bm{\Delta})\leq 0$.\\
\end{description}
\end{theorem}      
                
Let us make a few remarks. In the proof of items \rm{(i)} and \rm{(ii)} we repeatedly apply a perturbation result in order to handle non-small perturbations. This is not possible in items \rm{(iii)}\rm{(a) -- (c)} and item \rm{(iv)} because every perturbation in opposite direction of the cutset makes the graph  strongly connected and thus qualitatively changes the network's structure. In item (iii) a), the perturbation can be realised by turning an arbitrary node in the slave component into a hub having directed connections to all the nodes in the master component.
As we have shown numerically in the example in Figure \ref{Fig_example1} and also in \cite{pade2015}, not all perturbations in opposite direction of the cutset are increasing the synchronizability when $\bm{L}_1$ does not have zero column-sums. This is stated in item \rm{(iv)}: an example where the situation described in this item occurs is when the master component is an undirected sub-network, i.e when $L_1$ is symmetric, in which case all the eigenvalues are nonnegative.
%
%

In items \rm{(ii)-(iv)} we assume that the spectral gap is an eigenvalue of $\bm{L_{2}}+\bm{D_{C}}$. This happens for instance when the entries in the cutset $\bm{C}$ are very small (see Lemma \ref{lem:PositivOfEvectors}).
Topologically, this means that  the master component is very well connected in comparison to the intensity and/or density of the driving force. It is worth remarking that the connection density of the second component does not play a role in this scenario.
The rest of the paper is devoted to the proofs of our two main results and of other results completing our study.

\section[Motion of the spectral gap]{Synchronizability in Undirected Networks}\label{sec:Stab_Undir}

In this section we investigate weighted undirected networks and their behaviour under both, undirected and directed structural perturbations: our first main result (Theorem \ref{Thm_Dissec_of_graphs}) is proved here. The advantage of undirected networks is that tools from algebraic graph theory allow us to relate the coupling structure to algebraic properties of the associated graph Laplacian. We first look at the simpler case of undirected structural perturbations before proceeding to directed perturbations, i.e. the proof of Theorem \ref{Thm_Dissec_of_graphs}. 

\subsection{Undirected Structural Perturbations of Undirected Networks}
The following result was first established by M. Fiedler \cite{Fiedler1973} for unweighted graphs. Here, we present a straightforward generalization to weighted graphs.

\begin{proposition}\label{lem:FiedlerGrowth} 
Let $\mathcal{G}^{\prime}$ be a graph and $\mathcal{G}\subseteq\mathcal{G}^{\prime}$ a subgraph with corresponding Laplacians $\bm{L_W}$ and $\bm{L_{W^{\prime}}}$ respectively. Then we have
\[
\lambda_{2}\left(\bm{L_W}\right)\leq\lambda_{2}\left(\bm{L_{W^{\prime}}}\right).
\]
\end{proposition}
\begin{proof}
Using the Courant-Fischer theorem \cite{Horn1985} we have
\begin{eqnarray*}
\lambda_2(\bm{L_{W^{\prime}}})&=&\min_{\bm{v}\in\bm{1}^{\perp}}[\bm{v}^T \bm{L_{W} v}+\bm{v}^T(\bm{L_{W^{\prime}}} -\bm{L_{W}})\bm{v}]\\
&\geq&\min_{\bm{v}\in\bm{1}^{\perp}}[\bm{v}^T \bm{L_{W}v}]+\min_{\bm{v}\in\bm{1}^{\perp}}[\bm{v}^T(\bm{L_{W^{\prime}}} - \bm{L_{W}})\bm{v}]\\
&\geq&\min_{\bm{v}\in\bm{1}^{\perp}}[\bm{v}^T \bm{L_{W}v}]\\
&=&\lambda_2(\bm{L_{W}}).
\end{eqnarray*}
\end{proof}

Hence, increasing weights or adding links in an undirected network will never decrease the synchronizability. Perturbations which leave the spectral gap invariant can and do exist though. However, identifying graphs $\mathcal{G}\neq\mathcal{G}^\prime$ with corresponding Laplacians $\bm{L_W}$ and $\bm{L_{W^\prime}}$ respectively for which $\lambda_2(\bm{L_W})=\lambda_2(\bm{L_{W^\prime}})$ is a highly nontrivial problem. Fortunately, using perturbation theory of eigenvalues (Lemma \ref{lem:EVGrowth}) we can identify structural perturbations which leave the spectral gap unchanged up to first order. This approach will also enable us to investigate the effect of directed perturbations of undirected networks in the proof of Theorem \ref{Thm_Dissec_of_graphs}.

\subsection{Proof of Theorem \ref{Thm_Dissec_of_graphs}}\label{subsec:DirUndir}

The following standard result from matrix theory is the technical starting point for the rest of this article \cite{Horn1985}. It allows us to determine the dynamical effect of structural perturbations up to first order in the strength of the perturbation.

\begin{lemma}[Spectral Perturbation \cite{Horn1985}]\label{lem:EVGrowth}
Let $\lambda$ be a simple eigenvalue of $\bm{L}\in\mathbb{R}^{N\times N}$
with corresponding left and right eigenvectors $\bm{u},\bm{v}$ and
let $\tilde{\bm{L}}\in\mathbb{R}^{N\times N}$. Then, for $\varepsilon$
small enough there exists a smooth family $\lambda\left(\varepsilon\right)$
of simple eigenvalues of $\bm{L}+\varepsilon\tilde{\bm{L}}$ with
$\lambda\left(0\right)=\lambda$ and
\begin{equation}
\lambda^{\prime}\left(0\right)=\frac{\bm{u}^T\tilde{\bm{L}}\bm{v}}{\bm{u}^T\bm{v}}.\label{eq:Eigenv_Pert}
\end{equation}
\end{lemma}

So far we have only considered undirected perturbations. If we allow for directed perturbations the motion of the spectral gap over the real line is not constrained to one direction any more. Even for undirected networks, so far there is no general classification of directed perturbations according to their impact on the spectral gap. However, using results about the structure of undirected networks we can identify classes of directed perturbations which preserve the property that the slope of the spectral gap is nonnegative. The following result is due to M. Fiedler \cite{Fiedler1975}.

\begin{theorem}\label{thm:FiedlerDecomposition}\emph{[Fiedler, 1975]}
Let $\mathcal{G}$ be a connected weighted graph with Fiedler vector $\bm{v}$. For $r\geq 0$ definte $M(r)=\left\{ i | v_i+r\geq 0\right\}$. Then the following holds \\
(i) The subgraph $\mathcal{G}_r\subseteq\mathcal{G}$ induced by the set of nodes $M(r)$ is connected. \\
(ii) If the Fiedler vector $\bm{v}$ fulfils $v_i\neq 0$ for all $1\leq i\leq n$, the set of edges $\{i,j\}$ such that $v_iv_j<0$ defines a cut $C\subset\mathcal{E}$ and the resulting two components are connected.  
\end{theorem}

This result enables us to prove Theorem \ref{Thm_Dissec_of_graphs}.

\begin{proof} [Proof of Theorem \ref{Thm_Dissec_of_graphs}]
We first mention that by the genericity assumption in the Theorem we have that $\lambda_2(\bm{L_W})$ is simple and the Fiedler vector $\bm{v}$ has nonzero entries (see Theorems 3.1 and 3.2 in \cite{Poignard2017}). As a consequence, we can apply Theorem \ref{thm:FiedlerDecomposition}.\\
ad (i). The first assertion is a consequence of Theorem \ref{thm:FiedlerDecomposition} (ii).
Indeed, let $\mathcal{N}$ be the set of nodes of $\mathcal{G}$. Choosing the subgraph $\mathcal{G}_1$ induced by the nodes
$M(0)=\{i\in\mathcal{N}|v_i\geq 0\}$ and $\mathcal{G}_2$  as its complement yields two  connected subgraphs. 
Now, by Lemma \ref{lem:EVGrowth} we have $s_{k\rightarrow l}=v_l^2-v_lv_k > 0$ as $v_k$ and $v_l$ have opposite signs.
Now assume there is another decomposition of the nodes $I_1\subset\mathcal{N}$ and $I_2\subset\mathcal{N}$ such that $s_{k\rightarrow l}>0$ for all $k\in I_1$ 
and $l\in I_2$. Then, necessarily there must be two nodes $i\in I_1$, $j\in I_2$ such that $v_i$ and $v_j$ have the same sign.
Otherwise, it would be the same decomposition as before. But then either $s_{i\rightarrow j}<0$ or $s_{j\rightarrow i}<0$ 
or $s_{i\rightarrow j}=s_{j\rightarrow i}=0$ which is a contradiction.\\
ad (ii).  Assume without restriction that $k,l\in\mathcal{G}_1$ and $v_l\geq v_k$. Then we have $s_{k\rightarrow l} = v_l(v_l-v_k)\geq 0$. For the opposite link we obtain equivalently $s_{l\rightarrow k}= v_k(v_k-v_l)\leq 0$ as all the involved $v_i$ are positive.\\
ad (iii) Consider again the (connected) subgraphs $\mathcal{G}_1$ and $\mathcal{G}_2$ induced by  $M(0)$ as in (i). Assume without restriction that the $v_i>0$ are numbered by $i=1,\ldots,n_1$ and form a decreasing sequence as well as the $v_i<0$ are numbered by $i=n_1+1,\ldots,n$ and form a decreasing sequence. Now, set $\varepsilon=\frac{1}{2}\min\{|v_i-v_j| | i,j\geq n_1+1, v_i\neq v_j \}$ and define the positive numbers $d_i=-(v_{n_1+i}-\varepsilon)$ for $1\leq i\leq n-n_1$. Then, by Theorem \ref{thm:FiedlerDecomposition} we have that the subgraphs $\mathcal{G}_{1i}$ induced by $M(d_i)$ form an increasing sequence of connected graphs. Let now $j>i$. For a node $k$ in $\mathcal{G}_{1i}\cap\mathcal{G}_2$ and a node $l$ in $\mathcal{G}_{1j}\setminus\mathcal{G}_{1i}$  we have $v_k>v_l$ by construction. So using Lemma  \ref{lem:FiedlerGrowth} again yields $s_{l\rightarrow k} = v_k(v_k-v_l)<0$ as all the involved $v_i$ are negative.
The second decomposition starting from the subgraph $\mathcal{G}_2$ induced by the nodes $\{i|v_i\leq 0\}$  is obtained considering the Fiedler vector $-\bm{v}$.
\end{proof}

\begin{exmp}
In order to illustrate Theorem \ref{Thm_Dissec_of_graphs} consider again the graph from Figure \ref{Example2} and let $\bm{v}$ be its Fiedler vector, i.e. the normalized eigenvector associated to the spectral gap. Then, the cut from item (i) is obtained by dividing nodes according to the sign of the corresponding entry in $\bm{v}$. In Figure \ref{Example2} it corresponds to red and blue nodes. For the complete picture it remains to determine what exactly is the effect of perturbing among nodes of the same type of color beyond item (ii). This is done by comparing the corresponding entries in $\bm{v}$, color coded in Figure \ref{Fig_DissecGraphs}. Then, item $(iii)$ states that increasing weights or adding weak connections from dark red to light red nodes decreases the synchronizability, while adding an opposite link increases the synchronizability. The same is valid for the blue nodes.

\begin{figure}[H]
\includegraphics[scale=0.5]{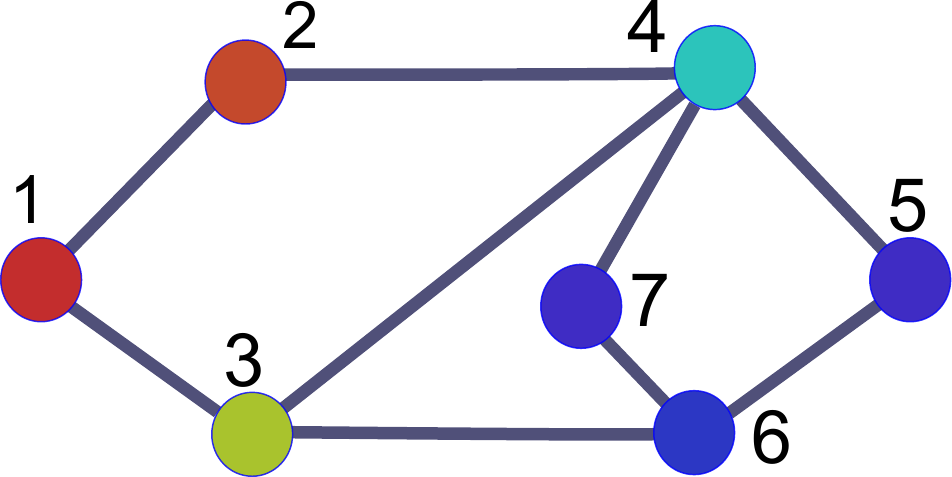}
\vspace{0.4cm}\caption[~~~~Dissection of graphs.]{Illustration of Theorem \ref{Thm_Dissec_of_graphs}. The subgraphs $\mathcal{G}_1$ and $\mathcal{G}_2$ are given by the nodes $\{1,2,3\}$ and $\{4,5,6,7\}$ respectively The Fiedler vector is given by $\bm{v}=(0.57,0.5,0.09,-0.06,-0.38,-0.34,-0.38)$ and the nodes are color coded with shades of red ($v_i>0$) and blue ($v_i<0$).}
\label{Fig_DissecGraphs}
\end{figure}
\end{exmp}

\subsection{Perturbations Leaving Invariant the Spectral Gap at First Order}

In this section we characterize perturbations which have no impact on the first order term of the spectral gap.

\begin{lemma}\label{lem:IdEvectors}
Let $\mathcal{G}$ be a connected weighted graph with Fiedler vector $\bm{v}$. Then:\\
(i) $v_k=v_l \iff s_{kl}=0 $.\\
(ii) $v_k=v_l \Rightarrow s_{k\rightarrow l}=s_{l\rightarrow k}=0$.
\end{lemma} 

\begin{proof}
Applying Lemma \ref{lem:EVGrowth} together with Equation (\ref{PertMatrix}) we obtain for an undirected perturbation $s_{kl}=(v_k-v_l)^2$ and hence the desired equivalence. For a directed perturbation we have similarly $s_{k\rightarrow l}=v_l(v_l-v_k)$.
\end{proof}
In some cases we can relate the algebraic property of having two identical entries in the Fiedler vector to the topology of the underlying graph.

\begin{proposition}\label{30June2017}
Let $\mathcal{G}$ be a connected weighted graph with adjacency matrix $\bm{W}$. Assume that the nodes $k$ and $l$ have the same set of adjacent nodes, i.e. $W_{ki}=W_{li}$ for all $1\leq i\leq N$.  Furthermore, assume that we have $d_k>d_{\min}$ where $d_{\min}$ is the minimal degree in $\mathcal{G}$. Then we have $s_{kl}=s_{k\rightarrow l}=s_{l\rightarrow k}=0$.
\end{proposition}

\begin{proof}
By the previous Lemma we only have to show that the Fiedler vector $\bm{v}$ fulfils $v_k=v_l$. Now, $\bm{v}$ is given as a solution of the equation $\bm{L_Wv}=\lambda_2 \bm{v}$. The $k$-th equation of this system writes as $\lambda_2 v_k = d_kv_k-\sum_{i\neq k}W_{ki}v_i$.
By assumption we have $d_k=d_l=:d$. Subtracting the $k$-th and the $l$-th component of the eigenvector equation we obtain
\begin{eqnarray*}
\lambda_2(v_k-v_l)&=&d_kv_k-d_lv_l-\left(\sum_{i\neq k}W_{ki}v_i-\sum_{i\neq l}W_{li}v_i\right)\\
\iff \lambda_2(v_k-v_l)&=&d(v_k-v_l)-(W_{kl}v_l-W_{lk}v_k)\\
\iff 0&=&(d-\lambda_2+W_{kl})(v_k-v_l).
\end{eqnarray*} 
Furthermore
\begin{eqnarray*}
d-\lambda_2+W_{kl}&\geq& d-\frac{N}{N-1}d_{\min}+W_{kl}\\
&\geq& d_{\min}+1-\frac{N}{N-1}d_{\min}+W_{kl}\\
&=& 1-\frac{1}{N-1}d_{\min}+W_{kl}\\
&>&W_{kl}\geq 0,
\end{eqnarray*}
where the before last inequality holds as we have $d_{\min}<d\leq N-1$. Hence, we must have $v_k=v_l$.
\end{proof}

\begin{OBS}
 (i) As an example consider the graph in Figure \ref{Fig_DissecGraphs}. Here, adding a weak link between nodes $5$ and $7$ leaves the spectral gap invariant up to first order, i.e.  $s_{57}=s_{7\rightarrow 5}=s_{5\rightarrow 7}=0$.\\
(ii) To see that the condition $d_k>d_{\min}$ is necessary, consider the weighted graph on $5$ nodes given by the following adjacency matrix
\begin{equation*}
\smaller{
\bm{W}=\left(\begin{array}{ccccc}
0 & 1 & 2 & 0 & 1\\
1 &0 &0 &1 &0 \\
2 &0 &0 &1.5 &0 \\
0 &1 &1.5 &0 &1 \\
1 & 0& 0& 1& 0\\
\end{array}\right).}
\end{equation*}
Nodes $2$ and $5$ have the same set of connections. However, the spectral gap $\lambda_2=2$ is simple and the Fiedler vector is given by $\bm{v}=(0,-1,0,0,1)^T$. So $v_2\neq v_5$ although both nodes have the same connections. And indeed, $d_2=d_5=2=d_{\min}$.\\
(iii) The proof of Proposition \ref{lem:IdEvectors} shows that when two nodes share the same connections, the corresponding entries of the Fiedler vector are equal. However, the opposite is in general not true. Consider the graph on $5$ nodes given by the adjacency matrix 
\begin{equation*}
\smaller{
\bm{W}=\left(\begin{array}{ccccc}
0 & 1 & 0 & 0 & 0\\
1 &0 &1 &1 &1 \\
0 &1 &0 &0 &1 \\
0 &1 &0 &0 &1 \\
0 & 1& 1& 1& 0\\
\end{array}\right).}
\end{equation*}
The spectral gap $\lambda_2=1$ is simple and the Fiedler vector is given by $\bm{v}=(-3,0,1,1,1)^T$. Hence, node $4$ and node $5$ have the same entry in the Fiedler vector although they don't have the same connections.
\end{OBS}

\section{Synchronizability in Directed Networks: Proof of Theorem \ref{thm:CutsetPerturbation}}\label{sec:proof_cutset} 

To prove Theorem \ref{thm:CutsetPerturbation} our aim is again to apply Lemma \ref{lem:EVGrowth} in order to track the motion of the real part of the spectral gap. In order to do so we first investigate the structure of the eigenvectors of $\bm{L_W}$ in the following two auxilliary lemmata. First observe that the matrix $\bm{L}_{2}+\bm{D_{C}}$ is nonnegative diagonally dominant \cite{Berman1994,Horn1985}. This property enables us to find a Perron-Frobenius like result.

\begin{lemma}\label{lem:PositivOfEvectors}
Let $\bm{L_W}$ be as in Theorem \ref{thm:CutsetPerturbation}. Then, $\bm{L}_{2}+\bm{D_{C}}$ has a minimal simple, real and positive eigenvalue with corresponding positive left and right eigenvectors.
\end{lemma}

\begin{proof}
Let $s:=\max_{i}\left\{{\bm{D_{C}}} _{(i)}+\sum_{j\neq i}W_{ij}\right\} >0$, then
$\bm{N}=s\bm{I}-\left(\bm{L}_{2}+\bm{D_{C}}\right)$ is a nonnegative
matrix by definition of $s$. Furthermore it is irreducible as we
assumed that the component associated to $\bm{L}_{2}$ is strongly
connected. Then, by the Perron-Frobenius theorem \cite{Berman1994},
$\bm{N}$ has a maximal, simple and real eigenvalue $\Lambda$ with corresponding
positive left and right eigenvectors $\bm{\omega}$ and $\bm{\eta}$.
That is 
\begin{eqnarray*}
\bm{N\eta} & = & \Lambda\bm{\eta}
\end{eqnarray*}
yielding
\begin{eqnarray*}
\left(\bm{L}_{2}+\bm{D_{C}}\right)\bm{\eta} & = & \left(s-\Lambda\right)\bm{\eta}.
\end{eqnarray*}

As $\Lambda$ is the maximal eigenvalue and all the eigenvalues of $\bm{L}_{2}+\bm{D_{C}}$ are obtained by eigenvalues $\mu$ of $\bm{N}$
through $s-\mu$, we must have that $s-\Lambda$ is the minimal real
eigenvalue of $\bm{L}_{2}+\bm{D_{C}}$. Furthermore, the eigenvectors
are the same, so the left and right eigenvectors of $\bm{L}_{2}+\bm{D_{C}}$ corresponding to
$s-\Lambda$ are positive. As a consequence of the Gershgorin Theorem together with the strong connectivity of the second component, we have that $\bm{L}_{2}+\bm{D_{C}}$ is invertible (Corollary 6.2.9 in \cite{Horn1985}). Hence, $s-\Lambda\neq 0$. Furthermore, by the Gershgorin Theorem again, we have $s-\Lambda\geq 0$ and hence $s-\Lambda >0$.
\end{proof}

This Lemma shows that the spectral gap and the corresponding eigenvectors are real in this case. So when changing the coupling structure, the motion of $\lambda_{2}$ will be along the real axis by Lemma \ref{lem:EVGrowth}. Next, we investigate the structure of the eigenvectors of $\bm{L_W}$.

\begin{lemma}
\label{lem:ReducedEvectors}Let $\bm{L_W}$ be as in Theorem \ref{thm:CutsetPerturbation} and let the spectral gap $\lambda_{2}$ of $\bm{L_W}$ be an eigenvalue of $\bm{L_{2}}+\bm{D_{C}}$. Then, the eigenvalue is simple and the corresponding left and right eigenvectors of $\bm{L_W}$ have the form
\[
\left(\bm{w}\bm{C}\left(\bm{L}_{1}-\lambda_{2}\bm{I}\right)^{-1},\bm{w}\right)\quad,\quad\left(\bm{0},\bm{y}\right)
\]
where $\bm{w}$ and $\bm{y}$ are left and right eigenvectors of $\bm{L}_{2}+\bm{D_{C}}$.
\end{lemma}

\begin{proof}
Let $\left(\bm{v},\bm{w}\right)$ and $\left(\bm{x},\bm{y}\right)$ be left and right eigenvectors of $\bm{L_W}$ corresponding to $\lambda_{2}$. For the left eigenvector we have 
\begin{eqnarray*}
\bm{0} & = & \left(\bm{v},\bm{w}\right)\bm{L_W}-\lambda_2\left(\bm{v},\bm{w}\right)\\
 & = & \left(\bm{v}\left(\bm{L}_{1}-\lambda_{2}\bm{I}\right)-\bm{wC},\bm{w}\left(\bm{L}_{2}+\bm{D_{C}}\right)-\lambda_2\bm{w}\right).
\end{eqnarray*}
The second component of this equation yields that $\bm{w}$ is a left eigenvector of $\left(\bm{L}_{2}+\bm{D_{C}}\right)$. As $\lambda_{2}$ is simple by Lemma \ref{lem:PositivOfEvectors}, it is not an eigenvalue of $\bm{L}_{1}$, so the first component yields 
\[
\bm{v}=\bm{wC}\left(\bm{L}_{1}-\lambda_{2}\bm{I}\right)^{-1}.
\]
The equation for the right eigenvector is
\begin{eqnarray*}
\bm{0} & = & \bm{L_W}\left(\begin{array}{c}
\bm{x}\\
\bm{y}
\end{array}\right)-\lambda_{2}\left(\begin{array}{c}
\bm{x}\\
\bm{y}
\end{array}\right)\\
 & = & \left(\begin{array}{c}
\left(\bm{L}_{1}-\lambda_{2}\bm{I}\right)\bm{x}\\
-\bm{C}\bm{x}+\left(\bm{L}_{2}+\bm{D_{C}}\right)\bm{y}-\lambda_{2}\bm{y}
\end{array}\right).
\end{eqnarray*}
As $\bm{L}_{1}-\lambda_{2}\bm{I}$ is regular we have $\bm{x}=\bm{0}$.
The second component then yields that $\bm{y}$ is a right eigenvector
of $\bm{L}_{2}+\bm{D_{C}}$.
\end{proof}

\begin{proof} [Proof of Theorem \ref{thm:CutsetPerturbation}]
 We will use throughout the proof that the smallest eigenvalue of  $\bm{L}_2+\bm{D_C}$ is simple, real and positive by  Lemma \ref{lem:PositivOfEvectors}.\\
Ad (i). 
Let the nonnegative matrix $\varepsilon \bm{\Delta}$ be a small perturbation of the cutset, so the corresponding Laplacian writes
$$\bm{L}_p\left(\varepsilon \bm{\Delta}\right)=\left(\begin{array}{cc}
\bm{L}_{1} & \bm{0}\\
-\bm{C}-\varepsilon\bm{\Delta} & \bm{L}_{2}+\bm{D}_{\bm{C}}+\varepsilon \bm{D}_{\bm{\Delta}}
\end{array}\right).$$
By assumption $\lambda_2$ is an eigenvalue of $\bm{L}_1$. As the smallest eigenvalue of $\bm{L}_2+\bm{D_C}$ is simple we can apply Lemma \ref{lem:EVGrowth} and obtain the following formula for the perturbed smallest eigenvalue $\mu_1(\varepsilon)$ of $\bm{L}_2+\bm{D_C}+\varepsilon\bm{D_{\Delta}}$ 
\[
\mu_1^{\prime}(0) = \frac{\bm{w}^T\bm{D_{\Delta}y}}{\bm{w}^T\bm{y}}>0.
\]
The positivity holds true because by Lemma  \ref{lem:PositivOfEvectors} left and right eigenvectors $\bm{w}^T, \bm{y}$ of $\bm{L}_2+\bm{D_C}$ are positive and at least one entry of $\bm{D_{\Delta}}$ is positive. So we have $\Re(\lambda_2)<\mu_1(0)<\mu_1(\varepsilon)$, i.e. the spectral gap of the whole network is still given by $\lambda_2$. Now, the perturbed matrix $\bm{L}_2+\bm{D_C}+\varepsilon\bm{D_{\Delta}}$ is of the same form as $\bm{L}_2+\bm{D_C}$. Hence,  the above reasoning can be applied repeatedly in order to obtain the desired result.\\ 
\quad \\
Ad (ii). Let again $\varepsilon\bm{\Delta}$ be a small perturbation in direction of the cutset and let us write the perturbed Laplacian $\bm{L_p}\left(\varepsilon\bm{\Delta} \right)$ as
\begin{eqnarray*}
\bm{L_p}\left(\varepsilon\bm{\Delta}\right) & = & \bm{L_W}+\varepsilon\left(\begin{array}{cc}
\bm{0} & \bm{0}\\
-\bm{\Delta} & \bm{D}_{\bm{\Delta}}
\end{array}\right).
\end{eqnarray*}
Using Lemma \ref{lem:EVGrowth} we have for the spectral gap of the perturbed system
\[
s(\bm{\Delta}) =\frac{\bm{w}^T\left(\bm{D_{\Delta}}\bm{y}-\bm{\Delta}\bm{x}\right)}{\left(\bm{v}^T,\bm{w}^T\right)\left(\begin{array}{c}
\bm{x}\\
\bm{y}
\end{array}\right)}
\]
where $\left(\bm{v},\bm{w}\right)$ and $\left(\bm{x},\bm{y}\right)$ are the eigenvectors of $\bm{L_W}$. Now, from Lemma \ref{lem:ReducedEvectors} we have $\bm{x}=\bm{0}$ and so we obtain
\[
s(\bm{\Delta})=\frac{\bm{w}^T\bm{D_{\Delta}}\bm{y}}{\bm{w}^T\bm{y}}.
\]
By assumption $\bm{\Delta}$ and therefore $\bm{D_{\Delta}}$ is nonnegative.
Furthermore, Lemma \ref{lem:PositivOfEvectors} shows that $\bm{w}$
and $\bm{y}$ are positive, so $s(\bm{\Delta})$ is positive. By the same reasoning as in (i) we can perform such small perturbations repeatedly in order to obtain the result for any structural perturbation in direction of the cutset with arbitrarily large entries.\\

Ad (iii){\rm (a)}. For a small perturbation $\varepsilon\bm{\Delta}$ in opposite direction of the cutset the perturbed Laplacian writes as 
\[
\bm{L}_{p}\left(\varepsilon \bm{\Delta}\right)=\left(\begin{array}{cc}
\bm{L}_{1}+\bm{D}_{\varepsilon\Delta} & -\varepsilon\bm{\Delta}\\
-\bm{C} & \bm{L}_{2}+\bm{D}_{C}
\end{array}\right).
\]
Using Lemma \ref{lem:EVGrowth} and \ref{lem:ReducedEvectors} yields 
\begin{equation}
s(\bm{\Delta})=-\frac{ \bm{w}^T \bm{M (\Delta)y} }{\bm{w}^T \bm{y} },
\label{EigenvPert2}
\end{equation}
where $\bm{w}$ and $\bm{y}$ are left and right eigenvectors of  $\bm{L}_{2}+\bm{D_C}$ and 
\begin{equation}
\bm{M (\Delta)}=\bm{C}(\bm{L}_{1}-\lambda_{2}\bm{I})^{-1}\bm{\Delta}.
\label{OppositeOperator}
\end{equation}
Now, assume we can find a $\bm{\Delta}$ such that $\bm{\Delta y}=\bm{1}$. Then we would have
\begin{eqnarray*}
\bm{M (\Delta) y} &=& \bm{C}(\bm{L}_{1}-\lambda_{2}\bm{I})^{-1}\bm{1}\\
&=&-\frac{1}{\lambda_2}\bm{C}\bm{1}.
\end{eqnarray*}
Now, by Lemma \ref{lem:PositivOfEvectors}  the eigenvectors $\bm{w}$ and $\bm{y}$  are postitive and $\bm{C}$ is nonnegative with at least one positive entry. Hence
\begin{eqnarray*}
s(\bm{\Delta}) &=&\frac{1}{\lambda_2}\frac{\bm{w}^T\bm{C1}}{\bm{w}^T\bm{y}}\\
&>&0.
\end{eqnarray*}

So it remains to show that there exists a $\bm{\Delta}$ such that $\bm{\Delta y}=\bm{1}$. By Lemma \ref{lem:PositivOfEvectors}, $\bm{y}$ is a positive vector, so for any fixed  $1\leq k\leq m $ we can choose $\bm{\Delta}_{ik}=\frac{1}{y_k}$ for $1\leq i \leq N$  and zero elsewhere to obtain $\bm{\Delta y}=\bm{1}$.\\

Ad (iii){\rm (b)}. Here we first use a result proved in \cite{Poignard2017} on the structure of Laplacian spectra in the case of strongly connected digraphs. Such graphs admit a spanning diverging tree and therefore, by Theorem 6.6 in \cite{Poignard2017}, we have that for a generic choice of the nonzero weights of $\bm{L}_1$, the spectrum of this matrix is simple. Under this genericity assumption, we can thus suppose that $\bm{L}_1$ is diagonalizable.

Then let's consider a vector $\bm{\Delta y}$ (with nonnegative entries) decomposed in the basis of eigenvectors $\left(\bm{1},\bm{X}_2,\cdots,\bm{X}_N\right)$ of $\bm{L}_1$:
\begin{align}\label{MSone}
\bm{\Delta y}=\beta_1\bm{1}+ \sum_{k=2}^N \beta_k \bm{X}_k,
\end{align}
with the numbers $\beta_i$ being possibly in $\mathbb{C}$. Such a decomposition gives the relation:
\begin{align}\label{MStwo}
-\left(\bm{L}_1-\lambda_2 \bm{I}\right)^{-1}\bm{\Delta y}=\dfrac{\beta_1}{\lambda_2}\bm{1}-\sum_{k=2}^N \dfrac{\beta_k}{\alpha_k-\lambda_2}\bm{X}_k,
\end{align}
where the numbers $\alpha_k$ denote the eigenvalues of $\bm{L}_1$ sorted by increasing order with respect to their real part, (so that $\alpha_1=0$).
Notice the fraction in this expression is well defined, since by assumption we have $\Re\left(\alpha_k\right)>\lambda_2$ for any 
$k\geq 2$.


Let's consider $\bm{\Delta y}=\bm{e}_k$, where $\bm{e}_k$ denotes the $k$-th vector of the canonical basis of ${\mathbb{R}}^N$. 
We first remark that the corresponding values $\beta_1\left(\bm{e}_k\right)$ in the decomposition \eqref{MSone} satisfy the following relation 
\begin{align*}
\bm{1}=1\cdot \bm{1}= \sum_{k=1}^N\beta_1\left(\bm{e}_k\right) \bm{1}+\sum_{k=1}^N\sum_{j=2}^N\beta_j(\bm{e}_k)\bm{X}_j
\end{align*}
which directly gives us the relation $\sum_{k=1}^N\beta_1\left(\bm{e}_k\right)=1$.

So at least one of these values, say $\beta_1\left(\bm{e}_{k_0}\right)$, must be positive. Consequently, for any nonnegative matrix $\bm{\Delta}$ with $\bm{\Delta y}=\bm{e}_{k_0}$ and $\lambda_2$ small enough we get that the right hand side in Equation \eqref{MStwo} is positive and hence, $s(\bm{\Delta})$ from Equation \eqref{EigenvPert2} must be positive.
In other words, since the terms in the sum in Equation \eqref{MStwo} depend only on $\bm{L}_1$ and $\bm{e}_{k_0}=\bm{\Delta y}$, there exists a constant $\delta(\bm{L}_1,\bm{e}_{k_0})$ such that if $0<\lambda_2<\delta(\bm{L}_1,\bm{e}_{k_0})$ then $s(\bm{\Delta})>0$.
To conclude it suffices to consider the set $\mathcal{A}$ of integers $1\leq k\leq N$ such that $\beta_1\left(\bm{e}_k\right)>0$ and to set 
\begin{align*}
\delta(\bm{L}_1)=\min_{k \in \mathcal{A}}\delta(\bm{L}_1,\bm{e}_{k}).
\end{align*}

Since $\mathcal{A}$ contains $k_0$ it is nonempty and thus $\delta(\bm{L}_1)$ exists and is positive. Now consider the structural perturbations $\bm{\Delta}$ with one link in opposite direction of the cutset for which exist an integer $k$ in $\mathcal{A}$ such that $\bm{\Delta y}=\bm{e}_k$, i.e the structural perturbations with only one link in opposite direction of the cutset ending at a node $k$ belonging to $\mathcal{A}$:
such structural perturbations exist (since the vector $\bm{y}$ is positive), and for any such $\bm{\Delta}$ we have $s(\bm{\Delta})>0$ provided $0<\lambda_2<\delta(\bm{L}_1)$.\\

Ad (iii)\rm{(c)}. As in {(b)} assume again that $\bm{L}_1$ is diagonalizable. If moreover it is zero-column sum, then the basis of eigenvectors $\left(\bm{1},\bm{X}_2,\cdots,\bm{X}_n\right)$ of $\bm{L}_1$ satisfies

\begin{align*}
\forall k\geq 2,\,\,\sum_{i=1}^n {X_k}_{(i)}=0.
\end{align*}
Indeed, this can be directly seen by mutltiplying each eigenvector equation $\bm{L}_1\bm{X}_k=\alpha_k\bm{X}_k$ by the vector $(1,\cdots,1)$ to the left.\\
As a result we must have $\beta_1 \left(\bm{e}_k\right)=\frac{1}{n}$ for any $1\leq k \leq n$.
Now consider any nonnegative matrix $\bm{\Delta}$: if $y_i$ denotes the i-th entry of $\bm{y}$
in its decomposition in the canonical basis $(\bm{e}_1,\cdots,\bm{e}_n)$, we have
\begin{align*}
-\left(\bm{L}_1-\lambda_2 \bm{I}\right)^{-1}\bm{\Delta y}&=\sum_{i=1}^ny_i\left[\dfrac{\beta_1(\bm{e}_i)}{\lambda_2}\bm{1}-\sum_{k=2}^n \dfrac{\beta_k(\bm{e}_i)}{\alpha_k-\lambda_2}\bm{X}_k\right]\\
&=\sum_{i=1}^ny_i\left[\dfrac{1}{n\,\lambda_2}\bm{1}-\sum_{k=2}^n \dfrac{\beta_k(\bm{e}_i)}{\alpha_k-\lambda_2}\bm{X}_k\right].
\end{align*}
Since all the entries $y_i$ are positive by Lemma \ref{lem:PositivOfEvectors}, thus to get $s(\bm{\Delta})>0$ it suffices that all the terms in brackets are positive vectors. So it suffices that $\lambda_2$ is small enough compared to $\frac{1}{n}$ and compared to the sums $\sum_{k=2}^n \frac{\beta_k(\bm{e}_i)}{\alpha_k-\lambda_2}\bm{X}_k$. Since the family $(\beta_k(\bm{e}_i))_{\substack{2\leq k\leq n\\
1\leq i\leq n\\}}$ is finite, we get again (as in Ad (iii)\rm{(b)}) the existence of a constant $\delta(\bm{L}_1)>0$ such that, if $0<\lambda_2<\delta(\bm{L}_1)$ then for any structural perturbation $\bm{\Delta}$ in opposite direction of the cutset, we have $s(\bm{\Delta})>0$.\\

Ad (iv).  As in {(b)}  and {(c)} we can suppose $\bm{L}_1$ is diagonalizable. 
Since the entries of the cutset $\bm{C}$ are nonnegative, in virtue of Eq. (\ref{EigenvPert2}) it suffices to show that there is a $\bm{\Delta}$ such that some entry of $(\bm{L}_1 - \lambda_2 \bm{I} )^{-1} \bm{\Delta} \bm{y}$ is positive.
Assume $\bm{L}_1$ admits a positive eigenvalue $\alpha_k$: then any eigenvector of $\bm{L}_1$ associated to $\alpha_k$ is real. 
Let's choose one such eigenvector $\bm{X}_k$: if one entry of $\bm{X}_k$ is negative, we define 
\[
\mathcal{G}_{-}=\{1\leq i\leq n: {\bm{X}_k}_{(i)}<0\} \mbox{~~ and ~~} \mathcal{G}_{+}=\{1\leq i\leq n: {\bm{X}_k}_{(i)}>0\}
\]
and consider  
\begin{align*}
\beta_m=\max \{-{\bm{X}_k}_{(i)}, i \in \mathcal{G}_{-} \}  
\mbox{~~ and ~~ } \beta_M = \max\{ \bm{X}_{k_{(i)}}, i \in \{ 1, \dots, N\}  \} , 
\end{align*}
if the set $\mathcal{G}_{-}$ is empty (resp. $\mathcal{G}_{+}$) we set $\beta_m = 0$ (resp. $\beta_M = 0$). 
Moreover, we consider $\bm{\Delta y} = \beta_m \bm{1}  + \bm{X}_k$. In this way $\bm{\Delta y}$ is nonnegative and there is a $\bm{\Delta}$ that solves this equation. Thus
\begin{align*}
(\bm{L}_1 - \lambda_2 \bm{I})^{-1} \bm{\Delta}\bm{y} = \frac{1}{\lambda_2}\left[ - \beta_m \bm{1} + \frac{1}{\frac{\alpha_k}{\lambda_2} - 1} \bm{X}_k\right].
\end{align*}
%
Assume that $\beta_M$ is attained in the $i$th entry, so we obtain
that if 
$$
0<\alpha_k< \left( \frac{\beta_M}{\beta_m} +1 \right)\lambda_2,
$$ 
then  ${(\bm{L}_1 - \lambda_2 \bm{I})^{-1} \bm{\Delta y}}_{(i)}>0.$
Therefore, for a cutset $\bm{C}$ connecting to this entry, we have $s(\bm{\Delta})\leq0$, as desired.\\ 
If all entries of $\bm{X}_k$ are nonnegative, then $\bm{\Delta y} = \bm{X}_k$. 
This yields
\begin{align*}
(\bm{L}_1 - \lambda_2 \bm{I})^{-1} \bm{\Delta}\bm{y} = \left[ \frac{1}{\alpha_k - \lambda_2} \bm{X}_k\right],
\end{align*}
and any $\alpha_k > \lambda_2$ suffices  from which we get this time that for any choice of the cutset $\bm{C}$ we have $s(\bm{\Delta})<0$. 
\end{proof}

In Item (iv) of this theorem, the choice of the cutset $\bm{C}$ for which we hinder synchronization is not that sharp: indeed suppose only one entry in $\bm{X}_k$ is negative. Then, we can apply the same reasoning to the vector $-\bm{X}_k$, for which $n-1$ entries will be non negative. In this case the suitable cutsets $\bm{C}$ will be more numerous.

{\bf Illustration of Item (iv) (Hinderance of Synchronization).} Consider the directed network in Figure \ref{Fig_example1} (without the addition of links). Assume that all connections in the master network have strength $1/2< w<1$ and the connections in the slave network have strength $1$. Then the spectrum of the network  can be decomposed as 
$$
\sigma (L_W) = \{ 0, 2w, 3w \} \cup \{ 1, 3 \}.
$$
So the spectral gap $\lambda_2 = 1$ belongs to $\sigma(\bm{L}_2 + \bm{D}_{\bm{C}})$.
With the notation from the proof above we have $\alpha_2 = 2w$ and the corresponding eigenvector of $\bm{L}_1$ is given by $\bm{X}_2 = (-1 , 1 , -1)$. Also, the right eigenvector corresponding to $\lambda_2$ of  $\bm{L}_2 + \bm{D}_{\bm{C}}$ is $\bm{y} = (1, 1)$. Hence, considering 
$$
\bm{\Delta y} = \beta_m  \bm{1} + \bm{X}_2 = (0, 2, 0)
$$
this equation can be solved by introducing a single link from any node of the slave component to any 
node of the master component. However, in view of the cutset $\bm{C}$ that starts from node $2$ of the master component, only 
connections to node $2$ can give a contribution to $s(\bm{\Delta})$. Hence, we can choose
$$
\bm{\Delta} = 
\left( 
\begin{array}{cc}
0 & 0 \\
2 & 0 \\
0 & 0  
\end{array}
\right)
$$ 
that is, a single link from node $4$ to node $2$ as in the Figure \ref{Fig_example1} will cause hinderance of synchronization since $s(\bm{\Delta}) < 0$.

More generally, the eigenvector $\bm{X}_k$ provides a spectral decomposition
$$
\mathcal{G}_{-}=\{1\leq i\leq n: {\bm{X}_k}_{(i)}<0\} \mbox{~~ and ~~}
\mathcal{G}_{+}=\{1\leq i\leq n: {\bm{X}_k}_{(i)}>0\}.
$$
When the cutset is from only the nodes of a single component $\mathcal{G}_{-}$ or $\mathcal{G}_{+}$ then it might be possible to hinder synchronization. This suggests that to improve synchrony it is best to drive the slave by mixing multiple inputs from $\mathcal{G}_{-}$ and $\mathcal{G}_{+}$.

\section{Discussion}\label{sec:discussion}

In this paper we have investigated the effect of structural perturbations on the transverse stability of the synchronous manifold in diffusively coupled networks. Establishing a connection between topological properties of a network and its synchronizability has been a challenge for the last few decades. So far, most of the existing literature focuses on establishing correlations supported by numerical simulations. Here, we present a first step in proving rigorous results for both, undirected and directed networks.

The first part on undirected networks is based on tools from algebraic graph theory, namely properties of the Fiedler vector. An interesting question which we could not answer here is whether we can find the link which, among all possible links, increases the spectral gap the most. In the light of our analysis this would correspond to nodes $k,l$ maximizing $v_k-v_l$. It was a long standing hypothesis that this maximum is reached for the two nodes which are connected by the longest path among all shortest paths in the graph. While this hypothesis was recently proven to be wrong with a tree graph as a counter-example \cite{Evans2011} , we believe that the construction of the subgraphs $\mathcal{G}_{1j}$ in Theorem \ref{Thm_Dissec_of_graphs} $iii)$ might shed light into a weaker formulation of this hypothesis.

For directed networks we have investigated the behaviour of a network when its cutset is perturbed.
There is only scenario we did not investigate here: When the spectral gap is an eigenvalue of $L_1$, determining the effect of a perturbation in opposite direction of the cutset cannot be solved in the framework presented above. It is of course possible to write down a similar term as in Equation (\ref{OppositeOperator}). However, in this case it involves left and right eigenvectors  of $L_1$. One would thus need to investigate eigenvectors of Laplacians of strongly connected digraphs, and more precisely the signs of their entries. To our knowledge there have been no attempts to do so, yet.

Even more involved is the question whether there exists a classification of links according to their dynamical impact in strongly connected networks. To our knowledge, no results have been obtained for the general case so far either. This is also due to the fact, that there are few attempts to extend the approaches on undirected graphs due to M. Fiedler (see\cite{Fiedler1973}) to directed graphs. As shown here, related results would essentially improve our understanding of the dynamical impact of a link in directed networks.

\bibliographystyle{plain}
\bibliography{Ybibliography}

\end{document}